\documentclass[12pt]{amsart}

\usepackage{amssymb}

\usepackage{enumitem}

\usepackage{graphicx}

\makeatletter
\@namedef{subjclassname@2020}{%
  \textup{2020} Mathematics Subject Classification}
\makeatother

\usepackage[T1]{fontenc}
\newtheorem{theorem}{Theorem}[section]

\newtheorem{lemma}[theorem]{Lemma}

\theoremstyle{definition}
\newtheorem{definition}[theorem]{Definition}
\newtheorem{remark}[theorem]{Remark}

\numberwithin{equation}{section}

\begin{document}

%%%%% To ease editing, for IMPAN journals add:

\baselineskip=17pt

%%%%%%%%%%%%%%%%

\title[A generalized second main theorem for closed subschemes]{A generalized second main theorem for closed subschemes}

\author[L. Wang]{Liang Wang}

\address[Liang Wang]{Department of Mathematics\\ Nanchang University\\
Nanchang 330031, P. R. China}

\email{wangliang@email.ncu.edu.cn}

\author[T. B. Cao]{Tingbin Cao}\thanks{The corresponding author: Tingbin Cao}

\address[Tingbin Cao]{Department of Mathematics\\ Nanchang University\\
Nanchang 330031, P. R. China}

\email{tbcao@ncu.edu.cn (the corresponding author)}

\author[H. Z. Cao]{Hongzhe Cao}

\address[Hongzhe Cao]{Department of Mathematics\\ Nanchang University\\
Nanchang 330031, P. R. China}

\email{hongzhecao@ncu.edu.cn}

\date{}

\begin{abstract}
Let $Y_{1}, \ldots, Y_{q}$ be closed subschemes which are located in $\ell$-subgeneral position with index $\kappa$ in a complex projective variety $X$ of dimension $n.$ Let $A$ be an ample Cartier divisor on $X.$ We obtain that if a holomorphic curve $f:\mathbb C \to X$ is Zariski-dense, then for every $\epsilon >0,$
\begin{eqnarray*}
\sum^{q}_{j=1}\epsilon_{Y_{j}}(A)m_{f}(r,Y_{j})\leq_{exc} \left(\frac{(\ell-n+\kappa)(n+1)}{\kappa}+\epsilon\right)T_{f,A}(r).
\end{eqnarray*}This generalizes  the second main theorems for general position case due to Heier-Levin [AM J. Math. 143(2021), no. 1, 213-226] and subgeneral position case due to He-Ru [J. Number Theory 229(2021), 125-141]. In particular, whenever all the $Y_j$ are reduced to Cartier divisors, we also give a second main theorem with the distributive constant. The corresponding Schmidt's subspace theorem for closed subschemes in Diophantine approximation is also given.
\end{abstract}



\maketitle

\section{Introduction and main results}
 In 1933, H. Cartan \cite{Ca33} established the important second main theorem for linearly nondegenerate holomorphic mappings from $\mathbb {C}^m $ into
$\mathbb {P}^{N}(\mathbb {C})$ intersecting hyperplanes in general position. He also conjectured that it should be true for hyperplanes in subgeneral position. In 1983, Nochka \cite{No83} confirmed the Cartan's conjecture by introducing the Nochka weight and Nochka constant. In 2009, Ru \cite{Ru09} extended the Cartan's second main theorem to the case of Cartier divisors in complex projective variety as follows (For more background of Nevanlinna theory, we refer to \cite{ru}). Here we use notations from Nevanlinna theory which can be found in Section 2.

\begin{theorem}\label{T1.1}\cite{Ru09}
Let $X$ be a complex projective variety of dimension $n$ and $D_{1},\ldots,D_{q}$ be effective Cartier divisors on $X,$ located in general position.
Suppose that there exists an ample Cartier divisor $A$ on $X$ and positive integers $d_{j}$ such that $D_{j}\sim d_{j}A$ for $j=1,\ldots,q.$ Let
$f:\mathbb {C}\rightarrow  X$ be a holomorphic map with Zariski dense image. Then, for every $\epsilon > 0,$
 \begin{eqnarray*}
\sum^{q}_{j=1}\frac{1}{d_{j}}m_{f}(r,D_{j})\leq_{exc} (n+1+\epsilon)T_{f,A}(r),
\end{eqnarray*}
where $\leq_{exc}$ means the inequality holds for all $r \in \mathbb{R}^{\geq 0}$ outside a set of finite Lebesgue measure.
 \end{theorem}

In 2017, Ru and Wang \cite{RW17} considered the case of closed subschemes and gave the second main theorem as follows.\par

\begin{theorem}\label{T1.2}\cite{RW17}
Let $X$ be a projective variety. Let $Y_{1},\ldots,Y_{q}$ be closed subschemes of $X$ such that at most $\ell$ of the closed subschemes meet at any point
$x\in X.$ Let $A$ be a big Cartier divisor on $X.$ Let $f : \mathbb{C} \rightarrow X$ be a holomorphic curve with Zariski-dense image. Let
\begin{eqnarray*}
\beta_{A,Y_{j}} = \lim_{N \rightarrow \infty}\frac{\sum_{m=1}^{\infty}h^{0}(\tilde{X}_{j},N\pi^{*}_{j}A-mE_{i})}{Nh^{0}(X,NA)},~j= 0,\ldots,q,
\end{eqnarray*}
where $\pi_{j} :\tilde{X}_{j} \rightarrow X$ is the blowing-up of $X$ along $Y_{j},$ with associated exceptional divisor $E_{j}.$  Then,
for every $\epsilon > 0,$
 \begin{eqnarray*}
\sum^{q}_{j=1}m_{f}(r,Y_{j})\leq_{exc} \ell(\max_{1\leq j\leq q}\{\beta^{-1}_{A,Y_{j}}\}+\epsilon)T_{f,A}(r),
\end{eqnarray*}
where $\leq_{exc}$ means the inequality holds for all $r \in \mathbb{R}^{\geq 0}$ outside a set of finite Lebesgue measure.
\end{theorem}

When the closed subschemes $Y_{j} = y_{j}$ are distinct points in $X$ and dim$X=n,$ then one may take $\ell =1$ in Theorem \ref{T1.2}, Mckinnon and Roth \cite{MR15}
have shown that the Seshadri constants $\epsilon_{y_{j}}(A)$ and $\beta_{A,y_{j}}$ have the following relation:
$$\beta_{A,y_{j}}\geq \frac{n}{n+1}\epsilon_{y_{j}}(A).$$

In order to improve the condition  $D_{j}\sim d_{j}A$ in Theorem \ref{T1.1},  Heier and Levin \cite{HL21}  used the notion of Seshadri constant $\epsilon_{D_{j}}(L)$ (see Definition \ref{D2.1}), and obtained the following generalization of the second main Theorem for closed subschemes in general position.

 \begin{theorem}\label{T1.3}\cite[Theorem 1.8]{HL21}
 Let $X$ be a complex projective variety of dimension $n.$ Let $Y_{0},\ldots,Y_{q}$ be closed subschemes of $X$ in general position,
$f:\mathbb{C}\rightarrow X$ be a holomorphic map with Zariski dense image. Let $A$ be an ample Cartier divisor on $X.$ Then, for every $\epsilon > 0,$
\begin{eqnarray*}
\sum^{q}_{j=0}\epsilon_{Y_{j}}(A)m_{f}(r,Y_{j})\leq_{exc}(n+1+\epsilon)T_{f,A}(r),
\end{eqnarray*}
where $\leq_{exc}$ means the inequality holds for all $r \in \mathbb{R}^{\geq 0}$ outside a set of finite Lebesgue measure.
\end{theorem}

In 2021, He and Ru \cite{HR21} considered the $\ell$-subgeneral position case and generalized Theorem \ref{T1.3} by using somewhat different methods. In this paper, we continue to generalize them to the case of the index $\kappa$ of $\ell$-subgeneral position. The concept of the index $\kappa$ of $\ell$-subgeneral position by  Ji-Yan-Yu \cite{JYY19} is generalized from Cartier divisors to closed subschemes.\par

\begin{definition}\label{D2.6}Let $X$ be a projective variety of dimension $n$ which is defined over an arbitrary field $k$ with characteristic zero and let $Y_{1}, \ldots, Y_{q}$ be $q$ closed subschemes of $X.$ Let $\ell$ and $\kappa$ be two positive integers such that $\ell\geq\dim X\geq \kappa.$\par

(a). The closed subschemes $\{Y_{1}, \ldots, Y_{q}\}$  are said to be in general position in $X$ if for any subset $I\subset\{1, \ldots, q\}$ with $\sharp I\leq \dim X+1,$
$$codim\left(\bigcap_{i\in I} Y_{i}\cap X\right)\geq\sharp I.$$\par

(b). The closed subschemes $\{Y_{1}, \ldots, Y_{q}\}$  are said to be in $\ell$-subgeneral position in $X$ if for any subset $I\subset\{1, \ldots, q\}$ with $\sharp I\leq \ell+1,$
$$\dim\left(\bigcap_{i\in I}Y_{i}\cap X\right)\leq \ell-\sharp I.$$\par

(c). The closed subschemes $Y_{1}, \ldots, Y_{q}$ are said to be in $\ell$-subgeneral position with index $\kappa$ if $Y_{1}, \ldots, Y_{q}$ are in $\ell$-subgeneral position and for any subset $J\subset \{1,\ldots,q\}$ with $\sharp I \leq \kappa,$
$$codim \left(\bigcap_{i\in I}Y_{i}\cap X\right)\geq \sharp I.$$
\end{definition}

\begin{theorem}\label{T1.7}
Let $f:\mathbb{C}\rightarrow $X$ $ be a meromorphic map with Zariski dense image, where $X$ is a complex projective variety of dimension $n.$ Let $Y_{1},\ldots,Y_{q}$
be closed subschemes which is located in $\ell$-subgeneral position with index $\kappa$ in $X,$ and $\ell \ge n.$ Let $A$ be an ample Cartier divisor on $X.$ Then,
for any $\epsilon >0,$
\begin{eqnarray*}\label{E1.13}
\sum^{q}_{j=1}\epsilon_{Y_{j}}(A)m_{f}(r,Y_{j})\leq_{exc} \left(\frac{(\ell-n+\kappa)(n+1)}{\kappa}+\epsilon\right)T_{f,A}(r),
\end{eqnarray*}
where $\leq_{exc}$ means the inequality holds for all $r \in \mathbb{R}^{\geq 0}$ outside a set of finite Lebesgue measure.
\end{theorem}

To prepare for the proof of Theorem \ref{T1.7}, we need firstly to obtain a second main theorem for Cartier divisors with distributive constant, which is also an interesting topic independently. The concept of distributive constant was originally given by Quang\cite{Qu20}, and we extend it to the case of Cartier divisors as follows.\par

\begin{definition}\label{D1.5}
Let $X$ be a projective variety of dimension $n$ defined over over arbitrary field $k$ with characteristic zero and let $D_{1},\ldots,D_{q}$ be $q$ Cartier divisors of $X.$ The distributive constant $\Delta $ of
$\{D_{1},\ldots,D_{q}\}$ in $X$ is defined by

$$\Delta := \max_{\varGamma \subset \{1,\ldots,q\}}\frac{\sharp \varGamma}{n-dim((\cap_{j \in \varGamma}Supp D_{j})\cap X(\overline{k}))}.$$Here, we note that $dim~\emptyset = - \infty.$
\end{definition}

The second main theorem for holomorphic curves intersecting Cartier divisors with distributive constant in a complex projective variety is stated as follows.

\begin{theorem}\label{T1.6}
Let $X$ be a complex projective variety of dimension $n.$ Let $D_{1}, \ldots, D_{q}$ be Cartier divisors of $X$ with the distributive constant $\Delta$
in $X.$ Let $A$ be an ample Cartier divisor on $X.$ Let $f:\mathbb C \to X$ be a holomorphic curve with Zariski-dense image. Then, for every $\epsilon >0,$
\begin{eqnarray*}\label{E1.2}
\sum^{q}_{j=1}\epsilon_{D_{j}}(A)m_{f}(r,D_{j})\leq_{exc} (\Delta(n+1)+\epsilon))T_{f,A}(r),
\end{eqnarray*}
where $\leq_{exc}$ means the inequality holds for all $r \in \mathbb{R}^{\geq 0}$ outside a set of finite Lebesgue measure.
\end{theorem}

If $D_{1},\ldots, D_{q}$ are in $\ell$-subgeneral position with index
$\kappa$, then by Remark \ref{R3.3} we have $\Delta \leq \frac{\ell-n+\kappa}{\kappa}.$ Hence, Theorem \ref{T1.6} generalizes Theorem \ref{T1.1},  Ji-Yan-Yu's work  \cite[Theorem 1.2]{JYY19} and He-Ru \cite[Theorem 2.6]{HR21}.\par

Whether can Definition \ref{D1.5} be generalized to the closed subschemes in order to obtain a generalized second main theorem for closed subschemes with distributive constant? This is an open question.\par

The rest of this paper is structured as follows. In section 2, we briefly recall some basic notations and definitions about the closed subscheme and Nevanlinna theory. Some preliminary lemmas and the proof of Theorem \ref{T1.6} for Cartier divisors with distributive constant are given in section 3. Then Theorem \ref{T1.7} will be proved in section 4. According to the Vojta's dictionary (\cite{VP87, VP09}) of the analogy between Nevanlinna theory and Diophantine approximation, the Schmidt's subspace theorems will be given in section 5.\par

\section{Preliminaries}

\subsection{Seshadri constants}

Let $X$ be a projective variety and $Y$  be a closed subscheme of $X,$ corresponding to a coherent sheaf of ideals $\mathcal{I}_{Y}.$ Let $\mathcal{S} =
\oplus_{d \geq 0}\mathcal{I}_{Y}^{d}$ be the sheaf of graded algebras, where $\mathcal{I}_{Y}^{d}$ is the $d$-th power of $\mathcal{I}_{Y},$ with the
convention that $\mathcal{I}_{Y}^{0} = \mathcal{O}_{X}.$ Then $\tilde{X} := Proj \mathcal{S}$ is called the blowing-up of $X$ with respect to $\mathcal{I}_{Y},$
 or, the blowing-up of $X$ along $Y.$

Let $\pi:\tilde{X}\rightarrow X$ be the blowing-up along $Y.$ From Proposition II.7.13(a) in \cite{Har77}, the inverse image ideal sheaf
$\tilde{\mathcal{I}}_{Y} = \pi^{-1}\mathcal{I}_{Y}\cdot\mathcal{O}_{\tilde{X}}$ is an invertible sheaf on $\tilde{X}.$ Let $E$ be an effective  Cartier
divisor in$\tilde{X}$ whose associated invertible sheaf is the dual of $\pi^{-1}\mathcal{I}_{Y} \cdot \mathcal{O}_{\tilde{X}}.$

\begin{definition}\label{D2.1}
Let $Y$ be a closed subscheme of a projective variety $X.$ Let $\pi :\tilde{X} \rightarrow X$ be the blowing-up of $X$ along $Y.$ Let $A$ be a nef Cartier
divisor on $X.$ We define the Seshadri constant $\epsilon_{Y}(A)$ of $Y$ with respect to $A$ to be the real number

$$\epsilon_{Y}(A) = sup \{\gamma \in \mathbb{Q}_{\geq 0} | \pi^{*}A-\gamma E~is~\mathbb{Q}-nef\},$$
where $E$ is an effective Cartier divisor on $\tilde{X}$ whose associated invertible sheaf is the dual of $\pi^{-1}\mathcal{I}_{Y}\cdot \mathcal{O}_{\tilde{X}}.$
\end{definition}

\subsection{ The Weil function and its properties}

In this section, we briefly recall the definition of the Weil function and its properties.

\begin{definition}\label{D2.2}
Let $D$ be a Cartier divisor on a complex variety $X.$ A local Weil function for $D$ is a function $\lambda_{D}:(X\setminus suppD) \rightarrow \mathbb{R}
$ such that for all $x \in X$ there is an open neighborhood $U$ of $x$ in $X,$ a nonzero rational function $f$ on $X$ with $D\mid_{U} = (f),$ and a
continuous function $\alpha : U \rightarrow \mathbb{R}$ such that

$$\lambda_{D}(x) = -\log |f(x)|+ \alpha(x)$$
for all $x \in (U \setminus Supp D)$.

A continuous (fiber) metric $||\cdot||$ on the line sheaf $\mathcal{O}_{X}(D)$ determines a Weil function for $D$ given by

$$\lambda_{D}(x) = -\log ||s(x)||$$
where $s$ is the rational section of $\mathcal{O}_{X}(D)$ such that $D = (s).$ For example, the Weil function for the hyperplanes $H=\{a_{0}x_{0}+\ldots+a_{n}x_{n} = 0\}$ is given by

$$\lambda _{H}(x)= \log \frac{\max_{0\leq i \leq n}|x_{i}|\max_{0\leq i \leq n}|a_{i}|}{|a_{0}x_{0}+\ldots+a_{n}x_{n}|}.$$

The Weil functions with respect to divisors satisfy the following properties.

(a) Additivity: If $\lambda_{1}$ and $\lambda_{2}$ are Weil functions for Cartier divisors $D_{1}$ and $D_{2}$ on $X,$ respectively, then
$\lambda_{1}+\lambda_{2}$ extends uniquely to a Weil function for $D_{1}+D_{2}.$

(b) Functoriality: If $\lambda$ is a Weil function for a Cartier divisor $D$ on $X,$ and if $\phi :X' \rightarrow X$ is a morphism such that
 $\phi(X') \not\subset supp D,$ then $x \mapsto \lambda(\phi (x))$ is a Weil function for the Cartier divisor $\phi^{*}D$ on $X'.$

(c) Uniqueness: If both $\lambda_{1}$ and $\lambda_{2}$ are Weil functions for a Cartier divisor on $X,$ then $\lambda_{1}=\lambda_{2}+O(1).$

(d) Boundedness from below: If $D$ is an effective divisor and $\lambda$ is a Weil function for $D,$ then $\lambda$ is bounded from below.
\end{definition}

For the closed subschemes case, let $Y$ be a closed subscheme of a projective variety $X.$ Then one can associate a Weil function
 $\lambda_{Y}: X \setminus Supp Y \rightarrow \mathbb {R},$ well-defined up to $O(1)$. The following lemma indicates that a closed subscheme can be expressed by some Cartier divisors.\par

\begin{lemma}\label{L2.3}
\cite[Lemma 2.2]{Si87} Let $Y$ be a closed subscheme of a projective variety $X.$ There exist effective Cartier divisors $D_{1},\ldots,D_{\ell}$
such that

$$Y = \cap_{i=1}^{\ell}D_{i}.$$

\end{lemma}

\begin{definition}\label{D2.4}
Let $Y$ be a closed subscheme of a projective variety $X.$ We define the Weil function for $Y$ as

$$\lambda_{Y} = \min \{\lambda_{D_{1}},\ldots,\lambda_{D_{\ell}}\}+ O(1),$$
where $Y = \cap_{i=1}^{\ell}D_{i}$ (by Lemma \ref{L2.3}, such $D_{i}$ exist). Then there is a Weil function for the closed subscheme $\lambda_{Y} : X \setminus Supp Y
\rightarrow \mathbb{R},$ which does not depend on the choice of Cartier divisors.
\end{definition}

We briefly recall the natural operations on subschemes which are similar to the case of divisors, more details can be found
in \cite[section 2]{Si87}. Let $Y,Z$ be closed subschemes of $X$,

(i) The sum of $Y$ and $Z$, denoted by $Y+Z,$ is the subscheme of $X$ with ideal sheaf $\mathcal I_{Y} \mathcal{I}_{Z}$.\par
(ii) The intersection of $Y$ and $Z$, denoted by $Y \cap Z,$ is the subscheme of $X$ with ideal sheaf $\mathcal{I}_{Y}+\mathcal{I}_{Z}.$\par
(iii) The union of $Y$ and $Z$, denoted by $Y \cup Z$, is the subscheme of $X$ with ideal sheaf $\mathcal {I}_{Y} \cap \mathcal{I}_{Z}.$\par
(iv) Let $\phi : X' \rightarrow X $ be a morphism of projective varieties, the inverse image of $Y$ is the subscheme of $X'$ with ideal sheaf $\phi
^{-1}\mathcal{I}_{Y} \cdot \mathcal{O}_{X'},$ denoted by $\phi ^{*}Y.$\par

In addition, some properties of the Weil functions for closed subschemes are written as follows.

If $Y$ and $Z$ are two closed subschemes of $X$, and $\phi : X' \rightarrow X $ is a morphism of projective varieties,

(i) $\lambda_{Y \cap Z} = \min \{\lambda_{Y},\lambda_{Z}\}.$

(ii) $\lambda_{Y+Z} = \lambda_{Y}+\lambda_{Z}.$

(iii) If $Y \subset Z,$ $\lambda_{Y}\leq \lambda_{Z}.$

(iv) $\lambda_{Y}(\phi(x)) = \lambda_{\phi^{*}Y}(x).$

\begin{lemma}\label{L2.5}\cite{Si87}
Let $Y$ be a closed subscheme of $X,$ and let $\tilde X$ be the blowing-up of $V$ along $Y$ with exceptional divisor $E.$
Then for $P \in \tilde X \setminus Supp E,$
 $$\lambda_{Y}(\pi(P)) = \lambda_{E}(P)+ O_{v}(1).$$

\end{lemma}

\subsection{ Characteristic function and proximity function}

Let $X$ be a complex projective variety  and $f: \mathbb{C}\rightarrow X $ be a holomorphic map. Let $ L \rightarrow X $ be a positive line bundle. Denote by $||\cdot||$ a Hermitian metric in $L$ and by $\omega$ its Chern form. We define the characteristic function of $f$ with respect
to $L$ by
$$T_{f,L}(r) = \int_0^{r}\frac{dt}{t}\int_{|z|<t}f^{*}\omega.$$
Since any line bundle can be written as $L = L_{1}\otimes L_{2}^{-1}$ with $L_{1}, L_{2}$ are both positive, we define $T_{f,L}(r) =
T_{f,L_{1}}(r)-T_{f,L_{2}}(r).$

The characteristic function satisfies the following properties:

(i) Functoriality: If $\phi :X \rightarrow X'$ is a morphism and if $L$ is a line bundle on $X',$ then
$$T_{f,\phi^{*}L} = T_{\phi \circ f,L}(r) +O(1).$$

(ii) Additivity: If $L_{1}$ and $L_{2}$ are line bundle on $X,$ then
$$T_{f,L_{1}\otimes L_{2}(r)}= T_{f,L_{1}}(r) +T_{f,L_{2}}(r)+O(1).$$

(iii) Positivity: If $L$ is positive and $f: \mathbb{C}\rightarrow X $ is non-constant, then
$$T_{f,L}(r) \rightarrow + \infty ~~as~~ r\rightarrow +\infty.$$

For an effective divisor $D$ on $X,$ we define a proximity function of $r,$ for any holomorphic map $f : \mathbb{C}\rightarrow X$ with
$f(\mathbb{C})\not \subset Supp D,$
$$m_{f}(r,D) := \int_0^{2\pi}\lambda_{D}(f(re^{i\theta}))\frac{d\theta}{2\pi}.$$
Analogously, for a closed subscheme $Y$ on $X,$ we define a proximity function of $r,$ for a holomorphic curve $f: \mathbb{C}\rightarrow X $ with $f(\mathbb{C})
\not \subset Supp Y,$
$$m_{f}(r,Y) := \int_0^{2\pi}\lambda_{Y}(f(re^{i\theta}))\frac{d\theta}{2\pi}.$$

\section{Some lemmas for divisors and the proof of Theorem \ref{T1.6} }

The following is a reformulation of \cite[Lemma 3.1]{Qu20} by taking the logarithm of both sides of the inequality.\par

\begin{lemma}\label{L3.3}\cite[Lemma 3.1]{Qu20}
Let $t_{0},t_{1},\ldots,t_{n}$ be $n+1$ integers such that $1 = t_{0} < t_{1} < \ldots <t_{n},$ and let $\Delta=\max_{1\leq s \leq n}\frac{t_{s}-t_{0}}{s}.$
 Then for arbitrary real numbers $a_{0}, a_{1}, \ldots, a_{n-1}$ with $a_{0} \geq a_{1} \geq \ldots \geq a_{n-1} \geq 1,$ we have
$$\sum^{n-1}_{i=0}(t_{i+1}-t_{i})\log a_{i} \leq \Delta \sum^{n-1}_{i=0}\log a_{i}.$$
\end{lemma}

Next we need to introduce the definition of $(t_{1}, t_{2}, \ldots,t_{n})$-subgeneral position for Cartier divisors, which was originally given by Quang\cite{Qu20} for hypersurfaces.\par

\begin{definition}\label{D3.2}
Let $X$ be a projective variety of dimension $n$ defined over a number field $k$ and let $D_{1},\ldots,D_{q}$ be $q$ Cartier divisors of $X.$ We say that $D_{1},\ldots,D_{q}$
 are in $(t_{1},t_{2},\ldots,t_{n})$-subgeneral position with respect to $X$ if for every $1\leq s \leq n$ and $t_{s}+1$ Cartier divisors $D_{j_{0}},\ldots,D_{j_{t_{s}}},$
 we have

$$\dim\cap^{t_{s}}_{i=0}Supp D_{j_{i}}\cap X \leq n-s-1,$$
where $t_{0},t_{1},\ldots,t_{n}$ are integers with $0=t_{0}<t_{1}<\ldots<t_{n}.$

\end{definition}

\begin{remark}\label{R3.3}

(i) If ${D_{1},\ldots,D_{q}}$ are in $(t_{1},\ldots,t_{n})$-subgeneral position with respect to $X$, then their distributive constant in $X$
satisfies
$$\Delta=\max_{1\leq k\leq n}\frac{t_{k}}{n-(n-k)} = \max_{1\leq k\leq n}\frac{t_{k}}{k}.$$

(ii) If $D_{1},\ldots,D_{q}$ are in $\ell$-subgeneral position with index $\kappa$ with respect to $X,$ then one has
$$\dim\left(\cap^{k}_{j=1}D_{i_{j}}\right)\leq n-\kappa-(k-(\ell-n+\kappa-1))=\ell-k-1.$$
Hence, ${D_{1},\ldots,D_{q}}$ are in $(1, 2,\ldots, \kappa-1, \ell-n+\kappa, \ell-n+\kappa+1,\ldots, \ell-1, \ell)$-subgeneral position with respect to $X$ and
thus
\begin{eqnarray*}
 \Delta&\leq&\max\left\{\frac{1}{n-(n-1)},\frac{2}{n-(n-2)},\ldots,\frac{\kappa-1}{n-(n-\kappa+1)}, \ldots,\right.\,\,\\&&\left. \frac{(\ell-n)+\kappa}{n-(n-\kappa)},
\frac{(\ell-n)+\kappa+1}{n-(n-\kappa+1)},\ldots,\frac{\ell}{n}\right\}\\& =& \max\left\{1,\frac{\ell-n}{\kappa}+1\right\}\\&=&\frac{\ell-n+\kappa}{\kappa}.
\end{eqnarray*}
\end{remark}

The following lemma is just a special case of \cite[Lemma 3.2]{Qu20}. \par

\begin{lemma}\label{L3.4}\cite[Lemma 3.2]{Qu20}
Let $k$ be a number field. Let $X \subset \mathbb{P}^{M}_{k}$ be a projective variety of dimension $n.$ Let $H_{0},\ldots,H_{l}$ be hyperplanes in $X$
 which are in $\{t_{1},t_{2},\ldots,t_{n}\}$-subgeneral position on $X$ where $t_{0},t_{1},\ldots,t_{n}$ are integers with $0=t_{0}<t_{1}<\ldots<t_{n}=l.$
Let $L_{1},\ldots,L_{q}$ be the normalized linear forms defining $H_{1},\ldots,H_{q},$ respectively. Then there exist linear forms $L'_{1},\ldots,L'_{n+1}$ of $M+1$ variables such that,\par
(i) $L'_{0}=L_{0};$\par
(ii) For every $s \in \{1,\ldots, n\},$ $L'_{s} \in span_{k}(L_{0},\ldots,L_{t_{s}});$\par
(iii) Let $H'_{j},~j=0,\ldots,n,$ be the hyperplanes defined by $L'_{j},~j=0,\ldots,n.$ Then they are in general position on $X.$\par
\end{lemma}

We need also the following lemma \cite[Theorem 1.22]{Sh13}, which plays an important role in the arguments regarding dimension.\par

\begin{lemma}\label{L3.5}\cite{Sh13, HR21}
Let $X$ be a projective variety over an algebraic closed field $k,$ and $A$ be an ample Cartier divisor on $X.$ Let $F\subset X$ be a proper irreducible subvariety. Then either $F\subset
 A$ or $\dim(F\cap A)\leq \dim F-1.$
\end{lemma}

For the proof of Theorem \ref{T1.7}, we need to prove the second main theorem for Cartier divisors with distributive constant. The basic idea is to consider the distributive constant for Cartier divisors by making use of methods of Heier-Levin\cite{HL21} and He-Ru\cite{HR21}. \par

\begin{proof}[Proof of Theorem \ref{T1.6}]
Fix a real number $\epsilon>0,$ choose a rational number $\delta>0$ such that
$$\delta\Delta+\delta\Delta(n+1+\delta)<\epsilon,$$
and for a sufficiently small positive rational number $\delta'$ depending on $\delta,$ $\delta A-\delta'D_{i}$
is $\mathbb{Q}$-ample for all $i=1,\ldots,q.$ By the definition of the Seshadri constant, there exists a rational number
 $\epsilon_{i}>0$ such that
$$\epsilon_{D_{i}}(A)-\delta' \le \epsilon_{i}\le \epsilon_{D_{i}}(A)$$
and that $A-\epsilon_{i}D_{i}$ is $\mathbb{Q}$-nef for all $i=1,\ldots,q.$ Then we have
$$(1+\delta)A-(\epsilon_{i}+\delta')D_{i}=(A-\epsilon_{i}D_{i})+(\delta A-\delta'D_{i})$$
is a $\mathbb{Q}$-ample divisor for all $i.$ Let $N$ be a large enough natural number such that $N(1+\delta)A$
and $N[(1+\delta)A-(\epsilon_{i}+\delta')D_{i}]$ become very ample integral divisors for all $i.$

We claim that if $\{D_{1},\ldots,D_{q}\}$ is in $(t_{1},t_{2},\ldots,t_{n})$-subgeneral position with respect to $X,$ then we can construct divisors $div(s_{i})$
on $X,$ $s_{i}\in H^0 (X,N(1+\delta)A-N(\epsilon_{i}+\delta')D_{i}),$ $i=1,\ldots,q,$ such that,\par

(i) $div(s_{i})\sim N(1+\delta)A,i=1,\ldots,q.$ \par
(ii) The divisors $div(s_{1}), \ldots, div(s_{q})$ are in $\{t_{1},t_{2},\ldots,t_{n}\}$-subgeneral position on $X.$\par

We define $div(s_{1}),\ldots,div(s_{q})$ by induction as follows. Assume that, for some $j\in \{1,\ldots,q\},$ $div(s_{1}),$ $\ldots,$ $div(s_{j-1})$
 with desired property have been defined and $div(s_{1}),\ldots,div(s_{j-1}),D_{j},\ldots,D_{q}$ are
in $\{t_{1},t_{2},\ldots,t_{n}\}$-subgeneral position on $X$ (for $j=1,$ this reduces to the hypothesis that
$D_{1},\ldots,D_{q}$ are in $\{t_{1},t_{2},\ldots,t_{n}\}$-subgeneral position).

By the definition of $\{t_{1},t_{2},\ldots,t_{n}\}$-subgeneral position,

 (a) For all $1\leq s\leq n, I' \subset \{1,\ldots,j-1\}, J' \subset \{j+1,\ldots ,q\},$ and $\sharp I'+\sharp J'=t_{s},$ we have
$$\dim (\cap_{i\in I'}div(s_{i})\cap(\cap_{k\in J'}D_{k})\cap X) \leq n-s,$$
and
$$\dim (D_{j}\cap(\cap_{i\in I'}div(s_{i}))\cap(\cap_{k\in J'}D_{k})\cap X) \leq n-s-1 .$$
We can find a non-zero section $s_{j}\in H^0 (X,N(1+\delta)A-N(\epsilon_{j}+\delta')D_{j})$
such that $s_{j}$ does not vanish entirely on any irreducible components of $\cap_{i\in I'}div(s_{i})\cap(\cap_{k\in J'}D_{k}),$
for all $1\leq s\leq n, I' \subset \{1,\ldots,j-1\}, J \subset \{j+1,\ldots ,q\},$ and $\sharp I'+\sharp J'=t_{s}.$ Then by Lemma \ref{L3.5}
\begin{eqnarray*}
&&\dim \{div(s_{j})\cap(\cap_{i\in I'}div(s_{i}))\cap(\cap_{k\in J'}D_{k})\cap X\} \\\nonumber&\leq& \max\left(\dim D_{j}\cap(\cap_{i\in I'}div(s_{i}))\cap
(\cap_{k\in J'}D_{k})\cap X,n-s-1 \right)\\\nonumber& = & n-s-1.
\end{eqnarray*}

  (b) For all $1 \leq s\leq n, I \subset \{1,\ldots,j-1\}, J\subset\{j+1,\ldots,q\},$ and $\sharp I +\sharp J = t_{s}+1,$
$$\dim (\cap_{i\in I}div(s_{i})\cap(\cap_{k\in J}D_{k})\cap X) \leq n-s-1.$$
From this, we get that $div(s_{1}),\ldots, div(s_{j-1}), div(s_{j}), D_{j+1}, \ldots, D_{q}$ are in $\{t_{1}, \ldots, t_{n}\}$-subgeneral position on $X.$
Thus, the divisors $div(s_{1}), \ldots, div(s_{q})$ satisfy the above required properties (i) and (ii). Hence, the claim is proved.\par

Denote $F_{i}=div(s_{i})$ for $i=1,\ldots, q.$ It's easy to see
that the distributive constant of $\{F_{1},\ldots,F_{q}\}$ is equal to the distributive constant of $\{D_{1},\ldots,D_{q}\}.$
Now we're going to prove the theorem.

It is sufficient for us to consider the case where $\Delta < \frac{q}{n+1}.$ Note that $\Delta \geq 1,$ and hence $q > n+1.$
If there exits $i \in \{1,\ldots,q\}$ such that $\cap^{q}_{j=1,j\ne i}F^{*}_{j}\cap X \ne \emptyset,$ then

$$\Delta \geq \frac{q-1}{n} \geq \frac{q}{n+1}.$$
This is a contradiction. Therefore, $\cap^{q}_{j=1,j\ne i}F^{*}_{j}\cap X = \emptyset$ for all $i \in \{1,\ldots,q\}.$

We denote by $\mathcal{I}$ the set of all permutations of the set $\{1,\ldots,q\}.$ Denote by $n_{0}$ the cardinality of $\mathcal{I},$ $n_{0} = q!$
and we write $\mathcal{I} = \{I_{1},\ldots,I_{n_{0}}\},$ where $I_{i}=\{I_{i}(0),\ldots,I_{i}(q-1)\} \in \mathbb{N}^q$ and $I_{1} <I_{2} <\ldots
< I_{n_{0}}$ in the lexicographic order.

For each $I_{i} \in \mathcal{I},$ since $\cap^{q-1}_{j=1}F_{I_{i}(j)} \cap X =\emptyset,$ there exists $n+1$ integers $t_{i,0},t_{i,1},\ldots,t_{i,n}$
with $0=t_{i,0} < \ldots < t_{i,n} =l_{i},$ where $l_{i} \leq q-2$ such that $\cap^{l_{i}}_{j=0}F_{I_{i}(j)} \cap X =\emptyset$ and

$$\dim (\cap^{s}_{j=0}F_{I_{i}(j)})\cap X = n-u~\forall ~t_{i,u-1}\leq s <t_{i,u},~1\leq u \leq n.$$
Then, $\Delta >\frac{t_{i,u}-t_{i,0}}{u}$ for all $1\leq u \leq n.$

It means that $\{F_{1},\ldots,F_{q}\}$ is in $(t_{i,1},t_{i,2},\ldots,t_{i,n})$-subgeneral position with respect to $X.$
Denote by $\phi: X \to \mathbb P^{\tilde{N}}(k)$ the canonical embedding associated to the very
ample divisor $N(1+\delta)A$ and let $H_{0},\ldots,H_{q-1}$ be the hyperplanes in $\mathbb P^{\tilde{N}}(k)$ with $F_{j}=\phi^*H_{j-1}$ for
$j=1,\ldots,q.$ We denote $L_{0},\ldots,L_{q-1}$ to be the linear forms defining $H_{0},\ldots,H_{q-1}$
respectively. By Lemma \ref{L3.4}, there exist hyperplanes $\hat H_{0},\ldots,\hat H_{n}$ with defining linear forms $\hat L_{0},\ldots,\hat L_{n},$
such that $\hat L_{0}=L_{0},$ and for every $s \in \{1,\ldots, n\},$ $\hat L_{s} \in span_{k}(L_{0},\ldots,L_{t_{s}}) $
and $\phi^*\hat H_{0},\ldots,\phi^*\hat H_{n}$ are located in general position on $X.$ Applying Theorem \ref{T1.1} to $\phi^*\hat H_{0},\ldots,\phi^*\hat H_{n},$
we conclude that there exists a Zariski-closed set $Z$ such that for all $x \in X(k) \backslash Z,$
 \begin{eqnarray*}
 \int_{0}^{2\pi}\sum^{n}_{i=0}\lambda_{\phi^*\hat H_{i}}(x)\frac{d\theta}{2\pi} \leq [(n+1)+\delta]T_{f,N(1+\delta)A}(r).
 \end{eqnarray*}

Consider a point $P=\phi (x) \in \mathbb{P}^{\tilde {N}}(k),$ fix a element $I_{i}\in \mathcal{I},$ we arrange such that
$$||L_{I_{i}(0)}(P)||\leq ||L_{I_{i}(1)}(P)||\leq \ldots \leq ||L_{I_{i}(q-1)}||,$$
which implies
 \begin{eqnarray}\label{E4.1}
\sum^{I_{i}(q-1)}_{j=I_{i}(0)}\lambda_{H_{j}(P)}\leq \sum^{t_{i,n}}_{j=t_{i.0}}\lambda_{H_{j}(P)}+O(1).
\end{eqnarray}
(The proof of (\ref{E4.1}) is similar to that of Lemma 20.7 in \cite{VP09} which is standard and is omitted here.)
Thus using the construction of $\hat L_{0},\ldots,\hat L_{n},$ we have
$$||\hat L_{u}(P)||\leq B\max_{0\leq j\leq t_{i,u} }||L_{j}(P)||=B||L_{t_{i,u}}(P)||$$
for all $u=0,\ldots,n$ and some constant $B>0.$ Thus, by the definition of Weil function,
$$\lambda_{\hat H_{u}}(P)\geq \lambda_{H_{t_{i,u}}}(P)+ O(1).$$
Therefore, by Lemma \ref{L3.3} and (\ref{E4.1}) we have
\begin{eqnarray}\label{E4.2}&&\sum^{q-1}_{i=0}\lambda_{H_{j}}(P)\\\nonumber&=& \sum^{I_{i}(q-1)}_{j=t_{i,0}}\lambda_{H_{j}}(P)\\\nonumber&\leq& \sum^{t_{i,1}}_{j=t_{i,0}}\lambda_{H_{j}}(P)+\sum^{t_{i,2}}_{j=t_{i,1}+1}
\lambda_{H_{j}}(P)+\cdots+\sum^{t_{i,n}}_{j=t_{i,n-1}+1}\lambda_{H_{j}}(P)+O(1)\\\nonumber&\leq&(t_{i,1}-t_{i,0})\lambda_{H_{t_{i,0}}}(P)+(t_{i,2}-t_{i,1})
\lambda_{H_{t_{i,1}}}(P)+\cdots\\\nonumber&&+(t_{i,n}-t_{i,n-1})\lambda_{H_{t_{i,n-1}}}(P)+\lambda_{H_{t_{i,n}}}(P)+O(1)\\\nonumber&\leq&\Delta
\sum^{n-1}_{j=0}\lambda_{H_{t_{i,j}}}(P)+\lambda_{H_{t_{i,n}}}(P)+O(1)\\\nonumber&\leq&\Delta\sum^{n}_{j=0}\lambda_{H_{t_{i,j}}}(P)+O(1)\\\nonumber
&\leq&\Delta\sum^{n}_{u=0}\lambda_{\hat H_{u}}(P)+O(1).
\end{eqnarray}

By the functoriality of Weil function, we have
$\lambda_{H_{j}}(P)=\lambda_{H_{j}}(\phi(x))=\lambda_{\phi^{*}H_{j}}(x)=\lambda_{F_{j+1}}(x),$ thus,
\begin{eqnarray*}
\int_{0}^{2\pi}\sum^{q}_{i=1}\lambda_{F_{i}}(x)\frac{d\theta}{2\pi}\nonumber&=&\int_{0}^{2\pi}\sum^{q-1}_{j=0}\lambda_{H_{j}}(P)\frac{d\theta}{2\pi}
\\\nonumber&\leq& \Delta\int_{0}^{2\pi}\sum^{n}_{i=0}\lambda_{\hat H_{i}}(P)\frac{d\theta}{2\pi}+O(1)\\\nonumber&\leq& \Delta[(n+1)+\delta]T_{f,N(1+
\delta)A}(r)+O(1).
\end{eqnarray*}

From the construction of divisors $F_{i}=div(s_{i})$ for $i=1,\ldots,q,$ we know that $F_{i}-N(\epsilon_{i}+\delta')D_{i}$ is effective for each $i=1,\ldots,q.$
By the definition of the Seshadri constant and the Boundedness from below of a Weil function,
$$\lambda_{F_{i}}(x)\geq N(\epsilon_{i}+\delta')\lambda_{D_{i}}+O(1)\geq N\epsilon_{D_{i}}(A)\lambda_{D_{i}}(x)+O(1).$$
Then,
\begin{eqnarray*}
N\int_{0}^{2\pi}\sum^{q}_{i=1}\epsilon_{D_{i}}(A)\lambda_{D_{i}}(x)\frac{d\theta}{2\pi}&\leq& \int_{0}^{2\pi}\sum^{q}_{i=1}\lambda_{F_{i}}(x)
\frac{d\theta}{2\pi}\\&\leq&\Delta[(n+1)+\delta]T_{f,N(1+\delta)A}(r)+O(1).
\end{eqnarray*}
Note that $T_{f,N(1+\delta)A}(r)=N(1+\delta)T_{f,A}(r),$ hence we get the inequality as follow,
$$\int_{0}^{2\pi}\sum^{q}_{i=1}\epsilon_{D_{i}}(A)\lambda_{D_{i}}(x)\frac{d\theta}{2\pi}\leq\Delta[(n+1)+\delta](1+\delta)T_{f,A}(r).$$
It implies that
$$\sum^{q}_{i=1}\epsilon_{D_{i}}(A)m_{f}(r,D_{i})\leq \Delta[(n+1)+\delta](1+\delta)T_{f,A}(r).$$
Recall the choice of $\delta$, we get for every $\epsilon >0,$
\begin{eqnarray*}
\sum^{q}_{j=1}\epsilon_{D_{j}}(A)m_{f}(r,D_{j})\leq_{exc} (\Delta(n+1)+\epsilon))T_{f,A}(r).
\end{eqnarray*}
\end{proof}

\section{Proof of theorem \ref{T1.7}}

Before the proof of Theorem \ref{T1.7}, we need a key lemma.

\begin{lemma}\label{L5.1}\cite[Lemma 5.4.24]{La17}
Let $X$ be a projective variety, $\mathcal{I}$ be a coherent ideal sheaf. Let $\pi:\tilde{X}\rightarrow X$ be the blowing-up of $\mathcal I$ with exceptional
divisor $E.$ Then there exists an integer $p_{0}=p_{0}(\mathcal I)$ with the property that if $p\geq p_{0},$ then $\pi_{*}\mathcal{O}_{\tilde {X}}(-pE)=
\mathcal{I}^{p},$ and moreover, for any divisor $D$ on $X,$
$$H^{i}(X,\mathcal{I}^{p}(D))=H^{i}(\tilde{X},\mathcal{O}_{\tilde{X}}(\pi^{*}D-pE))$$
for all $i\geq 0.$
\end{lemma}

We now prove the second main theorem for closed subschemes in $\ell$-subgeneral position with index $\kappa.$\par

\begin{proof}[Proof of Theorem \ref{T1.7}]
The proof basically follows Heier-Levin\cite{HL21} and He-Ru\cite{HR21}.
Denote by $\mathcal{I}_{i}$ the ideal sheaf of $Y_{i},$ $\pi_{i}:\tilde{X}\rightarrow X$ the blowing-up of $X$ along $Y_{i},$ and $E_{i}$
the exceptional divisor on $\tilde{X}_{i}.$ Then, fix a real number $\epsilon>0,$ choose a rational number $\delta>0$ such that
$$\frac{\delta(\ell-n+\kappa)}{\kappa}+\frac{\delta(\ell-n+\kappa)}{\kappa}(n+1+\delta)<\epsilon,$$
and for a sufficiently small positive rational number $\delta'$ depending on $\delta,$ $\delta A-\delta'E_{i}$
is $\mathbb{Q}$-ample for all $i=1,\ldots,q.$ By the definition of the Seshadri constant, there exists a rational number
 $\epsilon_{i}>0$ such that
$$\epsilon_{E_{i}}(A)-\delta' \le \epsilon_{i}\le \epsilon_{E_{i}}(A),$$
and that $A-\epsilon_{i}E_{i}$ is $\mathbb{Q}$-nef for all $i=1,\ldots,q.$ Then we have
$$(1+\delta)A-(\epsilon_{i}+\delta')E_{i}=(A-\epsilon_{i}E_{i})+(\delta A-\delta'E_{i})$$
is a $\mathbb{Q}$-ample divisor for all $i.$ Let $N$ be a large enough natural number such that $N(1+\delta)A$
and $N[(1+\delta)A-(\epsilon_{i}+\delta')E_{i}]$ become very ample integral divisors for all $i.$

We claim that if $\{Y_{1},\ldots,Y_{q}\}$ is in $\ell$-subgeneral position with index $\kappa$ with respect to $X,$ then we can construct divisors $F_{i}$
on $X,$ $i=1,\ldots,q,$ such that,\par

(i) $F_{i}\sim N(1+\delta)A,i=1,\ldots,q.$ \par
(ii) $\pi^{*}F_{i}\geq N(\epsilon_{i}+\delta')E_{i},i=1,\ldots,q.$ \par
(iii) The divisors $F_{1},\ldots,F_{q}$ are in $\ell$-subgeneral position with index $\kappa$ on $X.$\par

Like the special divisor case in the proof of preparation theorem, we can construct by induction. Assume that, for some $j\in \{1,\ldots,q\},$ $div(s_{1}),$ $\ldots,$
$div(s_{j-1})$ with desired property have been defined and $F_{1},$ $\ldots,$ $F_{j-1},$ $Y_{j},$ $\ldots,$ $Y_{q}$  $(F_{i}=div(s_{i}))$ are
in $\ell$-subgeneral position with index $\kappa$ on $X$ (for $j=1,$ this reduces to the hypothesis that
$Y_{1},\ldots,Y_{q}$ are in $\ell$-subgeneral position with index $\kappa$).
To find $F_{j}$, we let $\tilde{F}^{(j)}_{i}=\pi^{*}_{j}F_{i},i=1,\ldots,j-1$ and $\tilde{Y}^{j}_{i}=\pi^{*}_{j}Y_{i}$ for $i= j+1,\ldots,q. $
 Since in particular, $F_{1},\ldots,F_{j-1},Y_{j+1},\ldots,Y_{q}$ are in $\ell$-subgeneral position with index $\kappa$ on $X$, and
by noticing that $\pi^{-1}_{j}$ is an isomorphism outside of $Y_{j}$, we know that  $\tilde F_{1},\ldots,\tilde F_{j-1},\tilde Y_{j+1},\ldots,\tilde Y_{q}$ are
in $\ell$-subgeneral position with index $\kappa$ on $\tilde X_{j}$ outside of $E_{j}.$ Thus it reduces to the construction in the divisor
case, and by the argument in the divisor case, there are sections
$$ \tilde s_{j}\in H^0 (\tilde X_{j},\mathcal{O}_{\tilde{X}_{j}}(N(1+\delta)\pi^{*}_{j}A-N(\epsilon_{i}+\delta')E_{j})),$$
such that $\tilde F_{1},\ldots,\tilde F_{j-1},div(\tilde{s}_{j}),\tilde Y_{j+1},\ldots,\tilde Y_{q}$ are in $\ell$-subgeneral
position with index $\kappa$ on $\tilde X_{j}$ outside of $E_{j}$, where we regard $H^{0}(\tilde{X}_{j},\mathcal{O}_{\tilde{X}_{j}}(N((1+\delta)\pi^{*}_{j}A-(\epsilon_{j}+
\delta')E_{j})))$ as a subspace of $H^0 (\tilde X_{j},\mathcal{O}_{\tilde{X}_{j}}(N(1+\delta)\pi^{*}_{j}A)).$

To guarantee that $F_{j}$ has the required properties, by Lemma \ref{L5.1}, we have, for a big enough $N$,
$$H^0 (X,\mathcal{O}_{X}(N(1+\delta)A)\otimes \mathcal{I}^{N(\epsilon_{i}+\delta')}_{j})=H^{0}(\tilde{X}_{j},\mathcal{O}_{\tilde{X}_{j}}(N((1+\delta)\pi^{*}_{j}A-(\epsilon_{j}+
\delta')E_{j}))).$$
Therefore there is an effective divisor $F_{j}\sim N(1+\delta)A$ on $X$ such that $div(\tilde s_{j})=\pi^{*}_{j}F_{j}.$ Since $\tilde s_{j}\in H^0
(\tilde X_{j},\mathcal{O}_{\tilde{X}_{j}}(N(1+\delta)\pi^{*}_{j}A-N(\epsilon_{i}+\delta')E_{j})),$ we have $\pi^{*}F_{j}\geq N(\epsilon_{i}+\delta')
E_{j}$ on $\tilde{X}_{j}.$ $\tilde F_{1},\ldots,\tilde F_{j-1},div(\tilde{s}_{j}),\tilde Y_{j+1},\ldots,\tilde Y_{q}$ are in $\ell$-subgeneral
position with index $\kappa$ on $\tilde X_{j}$ outside of $E_{j},$ and $\pi_{i}$ is an isomorphism above the complement of $Y_{j},$ thus $ F_{1},\ldots, F_{j-1},F_{j},
Y_{j+1},\ldots,Y_{q}$ are in $\ell$-subgeneral position with index $\kappa$ on $X_{j}$ outside of $Y_{j}.$ Since $Y_{j}$ is in $\ell$-subgeneral position with index $\kappa$ with $F_{1},\ldots,F_{j-1},Y_{j+1},\ldots,Y_{q},$ it implies that $F_{1},\ldots,F_{j-1},F_{j},Y_{j+1},\ldots,Y_{q}$ are in $\ell$-subgeneral position with index $\kappa$ on $X.$
Thus, we obtain divisors $F_{1},\ldots,F_{q}$ with the required properties. Hence, the claim is proved.\par

Since $\Delta\leq\frac{\ell-n+\kappa}{\kappa},$ as a special case of the result in the proof of Theorem \ref{T1.6} in the above section, we can also get
\begin{eqnarray}\label{E5.1}\hspace{35pt}
\int_{0}^{2\pi}\sum^{q}_{i=1}\lambda_{F_{i}}(x)\frac{d\theta}{2\pi}\leq \left[\frac{\ell-n+\kappa}{\kappa}(n+1+\delta)\right]T_{f,N(1+\delta)A}(r)+O(1)
\end{eqnarray}
on $X(k)\setminus Z$ where $Z$ is a proper Zariski-closed subset of $X.$ By functoriality, additivity, and the fact that Weil functions of divisors are
bounded from below, for all $P\in \tilde {X}_{i}\setminus Supp E_{i},$
\begin{eqnarray*}
\lambda_{F_{i}}(\pi_{i}(P))&=&\lambda_{\pi^{*}_{i}F_{i}}(P)+O(1)\\\nonumber&\geq & N(\epsilon_{i}+\delta')\lambda_{E_{i}}(\pi_{i}(P))+O(1)\\\nonumber&
=& N(\epsilon_{i}+\delta')\lambda_{Y_{i}}(\pi_{i}(P))+O(1)\\\nonumber&\geq &N(\epsilon_{Y_{i}})\lambda_{Y_{i}}(\pi_{i}(P))+O(1).
\end{eqnarray*}
Together with (\ref{E5.1}),
$$\int_{0}^{2\pi}\sum^{q}_{i=1}N\epsilon_{Y_{i}}\lambda_{Y_{i}}(x)\frac{d\theta}{2\pi}\leq \left[\frac{\ell-n+\kappa}{\kappa}(n+1+\delta)\right]T_{f,N(1+\delta)A}(r)+O(1).$$
Then, by the choice of $\epsilon$, we have
$$\sum^{q}_{j=1}\epsilon_{Y_{j}}(A)m_{f}(r,Y_{j})\leq_{exc} \left(\frac{(\ell-n+\kappa)(n+1)}{\kappa}+\epsilon\right)T_{f,A}(r).$$
\end{proof}

\section{Schmidt's subspace theorem}

In this section, we give the counterpart in Diophantine approximation of our main results. The standard notations in Schmidt's subspace Theorem can be seen in \cite{HL21},\cite{HR21},\cite{VP87},\cite{VP09}).

Let $k$ be a number field. Denote by $M_{k}$ the set of places of $k$ and by $k_{v}$ the completion of $k$ for each $v \in M_{k}.$ Norms $||\cdot||_{v}$ on $k$ are normalized so that
$$||x||_{v}=|\sigma(x)|^{[k_{v}:\mathbb{R}]}\quad or \quad ||p||_{v}=p^{-[k_{v}:\mathbb{Q}_{p}]}$$
if $v\in M_{k}$ is an Archimedean place corresponding to an embedding $\sigma:k \to \mathbb{C}$ or a non-Archimedean place lying above the rational prime $p$, respectively.

An $M_{k}$-constant is a collection $(c_{v})_{v\in M_{k}}$ of real constants such that $c_{v}=0$ for all but finitely many $v$. Hights are logarithmic and relative to the number field used as a base field which is always denoted by $k.$ For $\mathbf{x}=(x_{0},\ldots,x_{n}) \in k^{n+1},$ define
$$||\mathbf{x}||_{v}:= \max\{||x_{0}||_{v},\ldots,||x_{n}||_{v}\},\quad v\in M_{k}.$$
The absolute logarithmic height of a point $\mathbf{x}=[x_{0}:\cdots:x_{n}] \in \mathbb{P}^{n}(k)$ is defined by
$$h(\mathbf{x}):=\sum_{v\in M_{k}}\log ||\mathbf{x}||_{v}.$$

For each $v\in M_{k},$ we can associate the local Weil functions $\lambda_{Y,v}$ which have similar properties as the Weil function introduced in Section 2.

Similar discussions as in Nevanlinna theory, one can easily obtain the counterparts of Theorem \ref{T1.6}, Theorem \ref{T1.7} for Schmidt's subspace theorems in Diophantine approximation. Thus we omit the details.

\begin{theorem}\label{T5.1}
Let $X$ be a projective variety of dimension $n$ defined over a number field $k.$ Let $S$ be a finite set of places of $k.$ For each $v\in S,$ let $D_{1,v}
,\ldots,D_{q,v}$ be Cartier divisors of $X$, defined over $k$, and with the distributive constant $\Delta$. Let $A$ be an ample Cartier divisor on $X$. Then, for $\epsilon > 0,$ there exists a Zariski-closed set $Z \subset X$ such that for all points $x \in X(k)\setminus Z$,
\begin{eqnarray*}
\sum_{v\in S}\sum^{q}_{j=1}\epsilon_{D_{j,v}}(A)\lambda_{D_{j,v},v} < (\Delta(n+1)+\epsilon))h_{A}(x).
\end{eqnarray*}
\end{theorem}

Theorem \ref{T5.1} generalizes Theorem 2.6 in \cite{HR21}.

\begin{theorem}\label{T5.2}
Let $X$ be a projective variety of dimension $n$ defined over a number field $k.$ Let $S$ be a finite set of places of $k.$ For each $v\in S,$ let $Y_{1,v}
,\ldots,Y_{q,v}$ be closed subschemes of $X$, defined over $k$, and in $\ell$-subgeneral position with index $\kappa$ in $X$, and $\ell \geq n.$ Let $A$ be
an ample Cartier divisor on $X$. Then, for $\epsilon > 0,$ there exists a Zariski-closed set $Z \subset X$ such that for all points $x \in X(k)\setminus Z$,
\begin{eqnarray*}
\sum_{v\in S}\sum^{q}_{j=1}\epsilon_{Y_{j,v}}(A)\lambda_{Y_{j,v},v} < \left(\frac{(\ell-n+\kappa)(n+1)}{\kappa}+\epsilon\right)h_{A}(x).
\end{eqnarray*}
\end{theorem}

Theorem \ref{T5.2} gives a generalization of Theorem 1.3 in \cite{HL21} and the Main Theorem (Arithmetic Part) in \cite{HR21}.

\subsection*{Acknowledgements} This work was supported by National Nature Science Funds of China (No. 11871260 and No. 12061041). The authors thank to the anonymous referee's for his or her suggestions to improve this paper.

\end{document}